\documentclass[12pt]{amsart}

\usepackage{graphicx}
\usepackage{latexsym,amsfonts}
\usepackage{mathrsfs}
\usepackage{epsf,psfrag,eepic,epic}
\usepackage{algorithm}
\usepackage{algorithmic}
\usepackage{colordvi}
\usepackage{array}
\usepackage{multirow}
\usepackage{url}
\usepackage{xcolor}
\usepackage{cleveref}

\newtheorem{theorem}{Theorem}[section]

\newtheorem{cor}[theorem]{Corollary}

\newtheorem{pbm}[theorem]{Question}

\theoremstyle{definition}
\newtheorem{definition}[theorem]{Definition}
\newtheorem{example}[theorem]{Example}
\newtheorem{rem}[theorem]{Remark}

\newcommand{\R}{\mathbb R}

\newcommand{\dif}{\mathrm{d}}

\newcommand{\Z}{\mathbb{Z}}
\newcommand{\N}{\mathbb{N}}
\newcommand{\C}{\mathbb{C}}

\newcommand{\Q}{\mathbb{Q}}

\newcommand{\weightsum}[3]{ L_{#1}(#2,#3)}

\begin{document}

\title{Sums of Weighted Lattice Points of Polytopes}

\author[De Loera]{Jes\'us A.\ De Loera}
\address{Department of Mathematics, University of California, Davis}
\email{deloera@math.ucdavis.edu}
\author[Escobar]{Laura Escobar}
\address{Department of Mathematics, Washington University in St. Louis}
\email{laurae@wustl.edu}
\author[Kaplan]{Nathan Kaplan}
\address{Department of Mathematics, University of California, Irvine}
\email{nckaplan@math.uci.edu}
\author[Wang]{Chengyang Wang}
\address{Department of Mathematics, University of California, Davis}
\email{cyywang@ucdavis.edu}

\maketitle

\centerline{\em Dedicated to Mich\`ele Vergne in celebration of her 80th birthday and} 
\centerline{\em with admiration of her broad and powerful contributions to Mathematics.}

\begin{abstract} We study the problem of counting lattice points of a polytope that are weighted by an Ehrhart quasi-polynomial of a family of parametric polytopes. As applications one can compute integrals and maximum values of such quasi-polynomials, as well as obtain new identities in representation theory. These topics have been of great interest to Mich\`ele Vergne since the late 1980's. Our new contribution is a result that transforms weighted sums into unweighted sums, even when the weights are very general quasipolynomials.  In some cases it leads to faster integration over a polytope.  We can create new algebraic identities and conjectures in algebraic combinatorics and number theory.
\end{abstract}

\section{Introduction}

We are given a rational convex polytope $P$ in $\R^n$ and $w(x)$ a quasi-polynomial function in $n$ variables. We call $w$ a \emph{weight function} (the precise definition of the quasi-polynomials we use here is below, but for now you can think of $w$ as a polynomial). A computational problem arising throughout the mathematical sciences, and of great interest to Mich\`ele Vergne, is to compute, or at least estimate,  

\begin{equation} \label{thesummation}
\weightsum{P}{w}{t}=\sum_{x\in tP \cap \Z^n}w(x).
\end{equation}

Note that, for a fixed $w$ as the polytope $P$ 
is dilated by an integer factor $t\in \N$  we obtain a function of $t$, which we call the \emph{weighted Ehrhart quasi-polynomial} for the pair $(P,w)$. The name is natural as when $w(x)=1$ then $\weightsum{P}{1}{t}$ yields the classical \emph{Ehrhart quasi-polynomial}.
We recommend \cite[Chapter 4]{Stanley1997Enumerative-Com} or \cite{beckrobins} and the references there for excellent introductions to Ehrhart functions and Ehrhart quasi-polynomials.

One can prove $\weightsum{P}{w}{t}$ is a quasi-polynomial in the sense that 
it is a function in the variable $t$ which is a sum of monomials up to degree $d+M$, where $M=\deg w$, but  whose coefficients $E_m $ are periodic functions of $t\in \N$: 
$$
\weightsum{P}{w}{t}= \sum_{m=0}^{d+M} E_m t^m.
$$
The leading coefficient of $\weightsum{P}{w}{t}$ is given by the integral of $w$ over the polytope $P$. These integrals were studied in \cite{Barvinok-1991}, \cite{Barvinok-1992} and more recently in \cite{BBDKV08}.

Sums of weighted lattice points are rather important as shown by their presence in so many areas of mathematics like enumerative combinatorics \cite{andrewspaulereise2001}, algebraic combinatorics \cite{FedericoErwan, chapotonqanalogues}, algebraic geometry \cite{Dickenstein+Vergneetal2012,torichochster} 
representation theory \cite{knutsontao1999, BaldoniVergneKronecker2018}, statistics \cite{DG, Chenetalstats},  and in symbolic integration and optimization \cite{BBDKV08,deloera-hemmecke-koeppe-weismantel:intpoly-fixeddim}, among many other fields. 

This paper has two goals.
First, we expand on a methodology for transferring weighted sums of lattice points inside a polytope to an unweighted sum of the lattice points (point enumeration) inside a new higher dimensional polytope.
Second, in our setting weights are very general (they can be quasi-polynomials) and thus our formulas find applications in the computation of integrals as well as in the generation of identities in number theory and representation theory. 


\subsection*{Our New Contributions:} 
We now outline the main contributions of this article. The main theorem is a surprisingly simple way to evaluate the function $\weightsum{P}{w}{t}$ where $P$ is a rational polytope and $w(x)$ is a very general weight function. The key idea is that we build a new polytope, the \emph{weight lifting polytope} $P^*$, for which these functions become simply $\weightsum{P^*}{1}{t}$, in other words, just a ``standard'' lattice point counting function. This way (often) the weighted Ehrhart polynomial $P$ is equivalent to the (usual) Ehrhart polynomial of $P^*$. Clearly, $P^*$ will depend on both $P$ and $w$:

\begin{theorem}[The existence of weight lifting polytopes] \label{polytopelifting} \quad
Let $P$ be a rational convex polytope in the form $\{\mathbf{x}\mid \mathbf{Ax} = \mathbf{b}, \mathbf{x}\geq 0\}$, where $\mathbf{A} \in \Z^{s \times n}, \mathbf{b} \in \Z^{s}$.
Let $Q(x_{1},\ldots,x_{n})$ be the parametric family of rational convex polytopes parameterized by $x_{1},\ldots,x_{n}$ as follows,
$$Q(x_{1},\ldots,x_{n}) = \left\{ \mathbf{y} \mid \mathbf{Cy} = \sum_{i=1}^{n}x_{i}\mathbf{d_{i}} + \mathbf{e}, \mathbf{y} \geq \mathbf{0} \right\},$$
where $\mathbf{C} \in \Z^{r \times m}, \mathbf{d_{i}},\mathbf{e} \in \Z^{r}$.  
Define $w(\mathbf{x})$ to be the multivariate Ehrhart quasi-polynomial function in $n$ variables that counts the number of lattice points in the parametric polytope $Q(x_{1},\ldots,x_{n})$ when $x_{i}$ are chosen integers, i.e., 
$$w(x_{1},\ldots,x_{n}) = |Q(x_{1},\ldots,x_{n}) \cap \mathbb{Z}^{m}|.$$
\begin{enumerate}
\item There is a \emph{weight lifting polytope} $P^*\subset \mathbb{R}^{n+m}$ defined by
$$P^{*} = \left\{ \left. \begin{pmatrix} \mathbf{x} \\ \mathbf{y} \end{pmatrix} \right|   \mathbf{A}^{*}\begin{pmatrix} \mathbf{x} \\ \mathbf{y} \end{pmatrix} = \begin{pmatrix} \mathbf{b} \\ -\mathbf{e} \end{pmatrix} ,\mathbf{x} \geq \mathbf{0}, \mathbf{y} \geq \mathbf{0} \right\}$$
where
$$\mathbf{A}^{*} = \begin{bmatrix} \mathbf{A} & \mathbf{0}  \\ \begin{matrix} \mathbf{d_{1}} & \mathbf{d_{2}} & \cdots & \mathbf{d_{n}} \end{matrix} & -\mathbf{C}  \end{bmatrix},
$$
for which the summation of the lattice points of $P$ weighted by $w$ equals the number of lattice points of $P^*$.

\item Moreover, when $\mathbf{e} = 0$, the construction is parametric in the sense that  the weight $w$ is a homogeneous function, $(tP)^{*} = t(P^{*})$, and 
$$\weightsum{P}{w}{t} = |(tP)^{*}\cap\mathbb{Z}^{n+m}| = |t(P^{*})\cap\mathbb{Z}^{n+m}| = \weightsum{P^*}{1}{t}.$$
\end{enumerate}
\end{theorem}

\begin{rem}
To the best of our knowledge the first weaker version of Theorem \ref{polytopelifting} appeared in print in work by Ardila and Brugall\'e (see \cite[Section 4]{FedericoErwan}). Unlike our paper, in \cite{FedericoErwan} the weights $w(x)$ were special polynomials, not quasipolynomials, and in that case some of the applications we show here were not possible.
In Section \ref{pruebas}, we present a direct constructive/algorithmic proof of Theorem \ref{polytopelifting} and describe several interesting special cases depending on the type of Ehrhart quasi-polynomials (in particular, we recover the results of \cite{FedericoErwan}). 
\end{rem}

\begin{rem}
Theorem \ref{polytopelifting} uses special weights that are by construction non-negative. But we note that most of the proof of the theorem works even when $w(x)$ takes negative or zero values over $P$. The function $\weightsum{P}{w}{t}$ still makes sense, but what we obtain is not a traditional Ehrhart polynomial, because, for example, the leading coefficient could be negative, and volumes are never negative. 
\end{rem}

\begin{rem}
 Theorem \ref{polytopelifting} says the weight $w(x)$ can be any Ehrhart quasi-polynomial. In Section \ref{pruebas}, we carefully discuss many ways to express polynomials in terms of these quasi-polynomial weights. A key point of our paper is that Theorem \ref{polytopelifting} is more versatile and expressive because it applies to more functions than just polynomial weights (a restriction in \cite{FedericoErwan}). In fact, in Section \ref{pruebas} we show that for one $w$ can have many different representations (e.g., polynomials), some more efficient than others. To demonstrate the power of our contributions Section \ref{applications} presents applications to Combinatorial Representation Theory and Number Theory. 
\end{rem}

Corollary \ref{integrate-maximize}  below is a notable new consequence of Theorem \ref{polytopelifting} that can be applied to many problems of interest. For example, these ideas can be applied to integration and maximization of Kostka numbers, Littlewood--Richardson coefficients, and any other combinatorial invariant that is given by an Ehrhart quasi-polynomial. 

\begin{cor} \label{integrate-maximize} Let $w$ be a weight obtained from an Ehrhart quasi-polynomial function of a parametric polyhedron $Q$, 
whose parameters are defined over the lattice points of a polytope $P$. Here $P,Q,w$ are just as in the second part of Theorem \ref{polytopelifting}. Using the weight lifting polytope construction of Theorem \ref{polytopelifting} one can integrate and maximize $w$ over $P$ as follows:
\begin{itemize}
\item One can compute the integral $\int_P w(x) dx$ reformulated as a volume computation of the weight lifting polytope $P^*$.

\item One can solve the maximization problem and determine $max_{\alpha \in P \cap \Z^n} w(\alpha)$. It reduces to counting the lattice points of a finite sequence of weight lifting polytopes which contain each other and can be read from $P^*$ efficiently. 
\end{itemize}
\end{cor}
We prove this result in Section \ref{pruebas}.
But first, we give a survey of weighted sums of lattice points inside a polytope and of the many contributions of Mich\`ele Vergne to this fascinating subject.

\subsection{Mich\`ele and Lattice Points in Polytopes} 
Mich\`ele Vergne has had several important contributions to the investigation of $L_P(w,t)$ which will be described briefly here.
We do this more or less in chronological order, but we highlight two main branches: structural results and applications.

First, let us review Mich\`ele's work on the analytic, algebraic, and combinatorial structure of weighted sums of lattice points. 

An initial inspiration was perhaps the early 1990's work by Khovanskii and co-authors \cite{Khovanskii+Pukhlikov1992ONE,Khovanskii+Pukhlikov1992TWO,Kantor+Khovanskii1993} who established special cases of the Euler-Maclaurin summation formula valid for certain \emph{unimodular lattice polytopes}, namely those for which  the primitive vectors on edges through each vertex of $P$ form a basis of the lattice. (These are sometimes called \emph{smooth polytopes} as they give smooth toric varieties. Flow polytopes are a key example.) The great contribution of Brion and Vergne was to generalize that work in their lovely papers \cite{brionvergne96french,brionvergne97}
to more general convex rational polytopes.

While we mostly work over the integer lattice and study simple 
dilations of a single polytope (as we explained in the introduction), all results can be stated much more generally. 
Consider an arbitrary lattice $L$ inside a finite-dimensional real vector space 
$V$ and let $P\subset V$ be a convex polytope with nonempty interior whose vertices are all rational with respect to $L$. A useful way to represent a polytope is in the form $P=\{ x: Ax=b, x\geq 0\}$ for a matrix $A$ and a vector $b$. Indeed, pick $A=(\alpha_k)_{1\leq k\leq N}$ where the columns $\alpha_i$ are a finite sequence of elements of $L$, all lying in the same open half-space or, equivalently, generating a pointed convex cone inside $V$. 
Note that for each $b\in L$ we have a (parametric) polytope 
$$P_A(b)=\{(x_1,\cdots,x_N)\in{\R}^N|\ \sum_{k=1}^Nx_k\alpha_k=b, \ x_k\geq0\}.$$ As $b$ changes we obtain many different types of polytopes! Observe that $P_A(b)$ is empty unless $b$ belongs to the cone generated by $\alpha_1, \cdots, \alpha_m$.

Besides the polytope $P_A(b)$ changing with $b$, it is important to note that for all $b \in L$, the equation $\sum_{k=1}^N x_k\alpha_k=b$ has only a finite number of non-negative integer solutions inside the polytope $P_A(b)$ for a fixed $A$. The function $L_A(b)$ counting the number of solutions depends on $b$, and it is called the \emph{vector partition function} associated to the matrix $A$. Geometrically $L_A(b)$ counts the number of lattice points inside the (parametric) polytope $P_A(b)$. 
We are interested in $L_A(b)$, and in the introduction we looked at one way to evaluate it.

As $b$ varies, the facets of $P_{A}(b)$ move parallel to themselves, so that the combinatorial type of the polytope $P_A(b)$ may 
change although there are only finitely many possible combinatorial types. 
Since the work of P. McMullen \cite{McMullen1978}, it is known that there is polyhedral complex, called the \emph{chamber complex}, that decomposes the cone $C(A)$ generated by the columns of $A$ as a union of finitely many polyhedral cones called \emph{chambers}. The chamber complex organizes the behavior of the number of integral points in a rational polytope when the facets of the polytope are translated, which corresponds to $b$ changing. Note that in the introduction we change $b$ by the simple (dilation) scaling $tb$. As we explain below the vector partition function $L_A(b)$ is given by a quasi-polynomial within each chamber (and in fact in a neighborhood of it as Mich\`ele and collaborators were able to show).

A related important function of $b$ is ${Vol}_A(b)$, denoting the volume of the polytope $P_A(b)$ normalized by choosing the Lebesgue measure 
on $V$ so that $V/L$ has volume 1 and ${\R}^N$ has the standard Lebesgue measure. Famously, Brion and Vergne investigated the volume and the partition functions. In fact, they expressed $L_A(b)$ in terms of the values of the function ${Vol}_A(b)$ and its derivatives at $b$. The volume function ${Vol}_A(b)$ is continuous and piecewise-polynomial in the variables $b$.  
The polynomial pieces are associated with each chamber. The function $L_A(b)$ is not continuous, but it is piece-wise quasi-polynomial in the same chambers.

The key idea was for Brion and Vergne to study the function $L_A(b)$ through the generating function $$L_A(y,b)=\sum_{x \in {\Z}^m_+\colon \ x_1 \alpha_1 + \cdots +x_m \alpha_m=b} \exp\{-\langle x, y \rangle\}$$ 
in $y=(y_1, \cdots, y_m) \in {\R}^m$
and its continuous version 
$${Vol}_A(y,b)=\int_{P_A(b)} \exp\{-\langle y, s \rangle\} \ ds$$ 
for $b, y \in {\R}^m$, where $\langle \cdot, \cdot \rangle$ is the scalar product on ${\R}^m$. They proved formulas for $L_A (y,b)$ and ${Vol}_A(y,b)$, expressing the quantities as finite sums of certain ``simple fractions'' at the vertices. Note that when $y=0$ we recover the usual volume and lattice point enumeration functions. Indeed one can use
\emph{Brion's cone decomposition theorem} \cite{Brion88}, which gives a conic decomposition of the generating functions of a polytope as a sum of rational functions
at each vertex. The rational functions are read directly from the tangent cones at each vertex.  These formulas depend both on the combinatorics (vertices and edges) of the polytope $P_A(b)$ as well as on their ``arithmetic'' complexity (the columns of $A$ and the lattice points inside fundamental parallelepipeds). In particular, the formulas contain some summations over finite Abelian groups that are factors of subgroups generated by $\alpha_k$ modulo the lattice $L$. The formulas are interpreted as residue formulas for functions of several complex 
variables as well as some identities of valuations and the polytope algebra \cite{barvinokzurichbook}. The connection between the two interpretations 
can be obtained via the Fourier transform, which transforms the characteristic functions of polyhedra into meromorphic functions (combinations of $\exp$ and rational functions). 

The computation of both $L_A(b)$ and ${Vol}_A(b)$ is based on generating-function methods; one tries to invert the Laplace transforms $$ \int_{ C(A) } {Vol}_A(b) \, e^{ - \left< b,z \right> } db = {1 \over { \prod_{ \alpha } \left< \alpha,z \right>}}, $$ and $$ \sum_{ b \in C(A) \cap { \Z}^n } L_A(b) \, e^{ - \left< b,z \right> } = {1 \over { \prod_{ \alpha } \left( 1 - e^{ - \left< \alpha,z \right> } \right) }}, $$ where $C(A)$ denotes the cone spanned by the column vectors of $A$ and the products are over the column vectors $\alpha$ of $A$. The Laplace transform inversion is done using Jeffrey-Kirwan residues (see \cite{ratsI,ratsII} and references there).

The study of such rational functions with poles over hyperplanes and 
polyhedral expressions was continued in several of Mich\`ele's papers, for example in \cite{ratsI,ratsII}, and in \cite{szenesvergne2003}
 where Szenes and Vergne considered a weighted extension of the (unweighted) vector partition function $L_A(b)$. They choose an exponential-polynomial weight function $w \colon {\R}^N \to {\C}$, that is, a linear combination of polynomials with exponential coefficients of the type $e^{ l(x)}$, where $ l\colon {\R}\to {\C}$ is a linear function, and consider the sum $L_A(w,b)$ of $w(x)$ over all integer points $x$ in the polytope $P_A (b)$. 
They showed that for a fixed $w$, the value of $L_A(w,b)$ as a function 
of $b$ is given by exponential-polynomial functions on the chambers of 
the cone generated by the columns of $A$. 
They established again a residue formula for $L_A(w,b)$ that allows them 
to compute it explicitly in some cases and show that the same exponential-polynomial function gives the correct formula when $b$ lies in a certain neighborhood of a chamber, even when the combinatorial type of the polytope $P_A(b)$ changes. 

Our paper deals with the situation when $w$ is a quasi-polynomial, but 
as we hinted in the introduction, the general theory that describes the \emph{weighted vector partition function} $L_A(w,b)$ as a piecewise quasi-polynomial function can be more directly explained for dilations of a polytope. Throughout the paper, we say that a polytope is rational if its vertices have rational coordinates. Let $P \subset {\Bbb R}^d$ be a rational polytope 
and let $w \colon {\Bbb R}^d \longrightarrow {\Bbb R}$ be a polynomial weight function. The sum $$ L_P(w,1)=\sum_{x \in P \cap {\Bbb Z}^d} w(x) $$ of the polynomial $w$ 
over the integer points $x$ in $P$ behaves as a quasi-polynomial when $P$ 
is dilated. That is, the function $t \longmapsto L_P(w,t)$ is a polynomial 
in $t \in {\Bbb N}$ with periodic coefficients. 

Berline and Vergne developed in \cite{Berline-Vergne-2007} a formula for the sum of the values of a polynomial weight function $w$ on $\R^d$ over the integer points in a rational convex polytope $P$. 
This takes the elegant form 
$$ \sum_{x\in P\cap\Bbb{Z}^d}\,w(x)=\sum_{f\in F(P)}\,\int_{f}\,D(P,f)\cdot w,$$ 
where $F(P)$ is the set of faces of $P$ and $D(P,f)$ is a differential operator of infinite order with constant coefficients on $\R^d$. For the constant function $h = 1$ (that is, the lattice point enumerator), the formula takes the form $$\# (P\cap\Z^d)=\sum_{f\in F(P)}\,\nu_0(P,f){\rm vol}(f), $$ where the coefficients $\nu_0(P,f)$ are rational numbers. This type of enumeration formula has been known since the work of McMullen and others (see e.g., the discussion in \cite{BarviPom,barvinokzurichbook}).
An important point of Mich\`ele's work (in contrast to other treatments) is that there is an algorithm which computes the $m$ terms of $D(P,f)$ of lowest order in polynomial running time with respect to the size of the data determining $P$, at least if $m$ and $d$ are fixed; this extends a well-known polynomial time algorithm of Barvinok \cite{barvinokehrhart94} for the lattice point enumerator. 

As we saw earlier the vector partition function $L_A(b)$ is a piecewise-defined quasi-polynomial, i.e., there exist polyhedral regions in ${\R}^d$, the chambers, such that $P_A(b)$ coincides with a fixed quasi-polynomial if $b$ is restricted to one chamber.    
However, if $b$ ``hits the wall'' between two adjacent chambers, so that the combinatorics of $P_A(b)$ changes, the vector partitition function does not have to remain  the same quasi-polynomial, but, interestingly, it often does. This is investigated in \cite{boysalvergne2012} where Boysal and Vergne expressed the difference of two such quasi-polynomials from adjacent regions of the chamber complex as a convolution of distributions. It was also considered later in \cite{wall-analytic2012}.

The leading term of the quasi-polynomials expressing $L_A(b)$ is the volume (in the unweighted case), or the integral of the weight function (in the weighted case), of the corresponding polytope $P_A(b)$. In the paper \cite{BBDKV08}, 
Mich\`ele and collaborators obtained optimal algorithms for the problem of integrating a polynomial function over a rational simplex (and thus a polytope). While they proved that the problem is NP-hard for arbitrary polynomials, on the other hand, if the polynomial depends only on a fixed number of variables, while its degree and the dimension of the simplex are allowed to vary, they showed that integration can be done in polynomial time. As a consequence, for polynomials of fixed total degree, there is a polynomial time algorithm as well. 

The work in \cite{BBDKV08} is only about the computation of volumes and integrals, i.e., the leading coefficients of the quasi-polynomials that described the vector partition function $L_A(w,b)$, but in \cite{BBDKV-highest2012} Mich\`ele and collaborators look at the coefficients of lower order terms of Ehrhart quasi-polynomials too. They showed that if $P$ is a simple polytope given by its vertices, then any number $k$, fixed in advance, of the highest terms of the quasi-polynomial $t \longmapsto L(tP, w)$ can be computed in polynomial time. Those terms are computed both as step-polynomials (polynomials in residue classes modulo some integers) as well as by their rational generating functions. 

The algorithm is based on the analysis of ``intermediate valuations''. These are sums of the form $$ \sum_{A=L+m } \int_{P \cap A} e^{\ell(x)} \ dx, $$ where the sum is taken over all lattice translates $A=L+m$, $m \in {\Bbb Z}^d$, of a fixed rational subspace $L \subset {\Bbb R}^d$ and $\ell \colon {\Bbb R}^d \longrightarrow {\Bbb R}$ is a fixed linear function. These intermediate valuations interpolate between exponential sums when $\dim L=0$ and exponential integrals when $\dim L=d$ and generalize similar intermediate objects considered by Barvinok. Intermediate sums or valuations were also studied by Mich\`ele and collaborators in several papers including \cite{inter1,inter2,3EhrhartQPs}.

The second branch of  Mich\`ele's work has included direct applications of the theory she and others developed. Indeed, we have discussed the structural theory of the vector partition functions $L_A(b)$, but mathematicians have in mind several crucial examples for applications. An important family of matrices $A$ corresponds to columns being  roots of type $A_n$ Lie algebras. These matrices correspond to the node arc incidence matrix  of directed networks and have important applications in optimization, statistics, and graph theory, and also in representation theory. The associated polytopes are the \emph{flow polytopes}.  The paper \cite{countingflows2004} 
discusses algorithms and software for the enumeration of all integral flows inside a network. The methods are again based on the study of rational functions with poles on arrangements of hyperplanes that appear in the $A$-type reflection groups. 
Similarly, in the paper \cite{BaldoniVergne-Kostant2008}
Baldoni and Vergne studied flow polytopes and the most famous connection to representation theory, the vector partition function of the directed acyclic complete graph, which is known as the Kostant partition function. They re-proved, using their rational function techniques, two famous results. First, they recover a formula discovered by Chan, Robbins, and Yuen (later proved by Doron Zeilberger). Then they generalize a formula due to Lidskiĭ relating the partition function to the volume.
In \cite{rootsvpartitions}, the authors apply these kinds of techniques in the more general setting of representation theory of Lie algebras, and the vector partition functions appear for all other classical root systems.

The connection to representation theory and lattice points continues to interest Mich\`ele. This is evident in the more recent papers \cite{BerlineVergneWalter2017, BaldoniVergneKronecker2018}. In the latest of these, the topic is
the decomposition of the symmetric algebra of ${\Bbb C}^N$ in terms of indecomposable $U(n_1)\otimes\cdots\otimes U(n_t)$-modules, which is determined by the Kronecker coefficients $g(\nu_1,\ldots,\nu_t)$ where $\nu_i$ is a partition of $n_i$. Baldoni and Vergne considered the dilation function ${k\to g(k\nu_1,\ldots,k\nu_t)}$ for an arbitrary fixed $t$-tuple of partitions. These functions are called \emph{dilated Kronecker coefficients}, and the theory of lattice point sums over polytopes gives an algorithm for computing them as a quasi-polynomial function of $k$. 
In the topic of representation theory, several algebraic invariants have a clear interpretation as the number of lattice points inside an alternating sum of several polytopes.

Finally, yet another application is in number theory, in the study of compositions of integers \cite{SylvesterDenumerant2015}. (We will see more on this exciting topic later in Section \ref{WENT}). For a given sequence  $\boldsymbol\alpha=[\alpha_1,\alpha_2,\dots, \alpha_{N+1}]$ of $N+1$ positive integers, one can  consider the combinatorial function $E(\boldsymbol\alpha)(t)$ that counts the non-negative integer solutions of the equation $\alpha_1x_1+\alpha_2x_2+\dots+\alpha_N x_N+\alpha_{N+1}x_{N+1}=t$, where the right-hand side $t$ is a varying non-negative integer. This is clearly a very special case of Ehrhart functions and in combinatorial number theory it is known as the \emph{Sylvester's denumerant}. Mich\`ele and collaborators used all the general theory in this case to give efficient algorithms to compute the highest $k+1$ coefficients of the Sylvester denumerant quasi-polynomial $E(\boldsymbol\alpha)(t)$ as step polynomials of $t$ (a simpler and more explicit representation). 

It is worth stressing that experimentation and concrete computation is crucial in this type of work and most implementations and experiments appeared in LattE integrale \cite{latte-1.2}.

\section{Proofs of Theorem \ref{polytopelifting} and sketch of other results} \label{pruebas}

Here we present proofs of Theorem \ref{polytopelifting} and some variations of it.

\begin{proof}[Proof of Theorem \ref{polytopelifting}]
Note that there is a natural projection map $  \pi: P^{*} \to P $ via $  (\mathbf{x}, \mathbf{y} ) \mapsto \mathbf{x}$. 
It suffices to show that for any fixed $\mathbf{x} \in P\cap \Z^{n}$, $w(\mathbf{x}) = |\pi^{-1}(\mathbf{x}) \cap \Z^{n+m}|.$
Recall that $(\mathbf{x}, \mathbf{y} ) \in \pi^{-1}(\mathbf{x})  $ if and only if $\mathbf{A}\mathbf{x} = \mathbf{b}$ and $\mathbf{Cy} = \sum_{i=1}^{n}x_{i}\mathbf{d_{i}} + \mathbf{e}$ where $\mathbf{x}\geq \mathbf{0},\mathbf{ y} \geq \mathbf{0}$. Given $\mathbf{x} \in P \cap \Z^{n}$, we see that  $(\mathbf{x},\mathbf{y} ) \in \pi^{-1}(\mathbf{x}) \cap \Z^{n+m} $ if and only if $ \mathbf{y} \in  Q(x_{1},\ldots,x_{n}) \cap \Z^{m}$. Hence, for a fixed $\mathbf{x} \in P \cap \Z^{n}$,
    \[ |\pi^{-1}(\mathbf{x}) \cap \Z^{n+m}|
    =  | Q(\mathbf{x}) \cap \Z^{m}|
    = w(\mathbf{x}).\]

We now consider second part of Theorem \ref{polytopelifting}. 
If $P = \{ \mathbf{x} : \mathbf{Ax} = \mathbf{b}, \mathbf{x} \geq \mathbf{0} \}$, then
$$P^{*} = \left\{ \left. \begin{pmatrix} \mathbf{x} \\ \mathbf{y} \end{pmatrix} \right|   \mathbf{A}^{*}\begin{pmatrix} \mathbf{x} \\ \mathbf{y} \end{pmatrix} = \begin{pmatrix} \mathbf{b} \\ -\mathbf{e} \end{pmatrix} ,\mathbf{x} \geq \mathbf{0}, \mathbf{y} \geq \mathbf{0} \right\}.$$
Therefore, $tP = \{ \mathbf{x} : \mathbf{Ax} = t\mathbf{b}, \mathbf{x} \geq \mathbf{0} \}$ and
$$(tP)^{*} = \left\{ \left. \begin{pmatrix} \mathbf{x} \\ \mathbf{y} \end{pmatrix} \right|   \mathbf{A}^{*}\begin{pmatrix} \mathbf{x} \\ \mathbf{y} \end{pmatrix} = \begin{pmatrix} t\mathbf{b} \\ -\mathbf{e} \end{pmatrix} ,\mathbf{x} \geq \mathbf{0}, \mathbf{y} \geq \mathbf{0} \right\}.$$
Given $\mathbf{e} = 0$, we can see that $(tP)^{*} =t(P^{*}).$
By the first part of the proof we conclude
\begin{equation*}
    \weightsum{P}{w}{t} = |(tP)^{*}\cap\mathbb{Z}^{n+m}| = |t(P^{*})\cap\mathbb{Z}^{n+m}| = \weightsum{P^*}{1}{t}. \qedhere
\end{equation*}
\end{proof}

Now we outline more results and corollaries of Theorem \ref{polytopelifting}  obtained by considering various weight functions. In order to do so we will be scaling polytopes by integers instead of only non-negative integers. Let us recall what this means.

\begin{definition}\label{def:Q}
Consider the rational polytope $Q = \{\mathbf{y} \mid \mathbf{C}\mathbf{y} = \mathbf{d}, \mathbf{y} \geq \mathbf{0}\} \subseteq \mathbb{R}^m$ where $\mathbf{C} \in \Z^{r \times m}, \mathbf{d} \in \Z^{r}$.
For every integer $t$, we define $tQ = \{\mathbf{y} \mid \mathbf{Cy} = t \mathbf{d}, \mathbf{y} \geq \mathbf{0}\}$.
\end{definition}
\begin{example}
Consider the $(m-1)$-dimensional standard simplex 
\[
\Delta_{m-1}= \{\mathbf{y} \mid y_{1}+\cdots+y_{m} = 1, y_{i} \geq 0 \}.
\] 
Then $-2 \Delta_{m-1} = \{\mathbf{y} \mid y_{1}+\cdots+y_{m} = -2\cdot 1, y_{i} \geq 0 \}$.
\end{example}
\begin{definition}
A function $w(t)$ is a \textbf{late-dilated} Ehrhart quasi-polynomial if
$$w(t) = |(t-c) Q \cap \Z^{m}|,$$
where $c\in\Z$ and $Q$ is a rational polytope of the form given in Definition \ref{def:Q}.
\end{definition}
\begin{example}
The function $\binom{t}{m-1}$ is a {late-dilated} Ehrhart polynomial in the variable $t$, because $\binom{t}{m-1}  = |(t-m+1) \Delta_{m-1} \cap \Z^{m}|$.
\end{example}

\begin{cor}\label{lifting} \quad
Let $w_{1},w_{2},\cdots,w_{n}$ be $n$ {late-dilated} Ehrhart quasi-polynomials, i.e., $w_{i}(t) = |(t-c_{i}) Q_{i} \cap \Z^{m_{i}}|$ where $Q_{i} = \{ \mathbf{y_{i}} \mid \mathbf{C_{i}y_{i}} = \mathbf{d_{i}}, \mathbf{y_{i}} \geq \mathbf{0} \}$ and $\mathbf{C_{i}} \in \Z^{r_{i} \times m_{i}}, \mathbf{d_{i}} \in \Z^{r_{i}}, c_{i} \in \Z$. Consider a rational polytope of the form $P = \{ \mathbf{x} \mid \mathbf{Ax} = \mathbf{b}, \mathbf{x} \geq \mathbf{0} \} \subseteq \mathbb{R}^n$ where $\mathbf{A} \in \Z^{s\times n}, \mathbf{b} \in \Z^{s}$ and the multivariate function $w( \mathbf{x} ) = \prod_{i=1}^{n} w_{i}(x_{i})$. 
There exists a \emph{weight lifting polytope} $P^{*}\subseteq \mathbb{R}^{n^*}$ of $P$, where $n^{*} = n + m_{1} + \cdots + m_{n}$,  such that
$$  \sum_{\mathbf{x} \in P \cap \Z^{n}}w(\mathbf{x}) =  \left| P^{*} \cap \Z^{n^{*}} \right|.$$
\end{cor}
\begin{proof}
We need only show that there is a rational polytope $Q(x_1,\ldots, x_n)$ of the form given in Theorem \ref{polytopelifting} for which $w(\mathbf{x}) = | Q(x_{1},\ldots,x_{n}) \cap \mathbb{Z}^{m_{1}+\cdots+m_{n}}|$ and then apply Theorem \ref{polytopelifting}.
Let $Q(x_{1},\ldots,x_{n}) = \prod_{i=1}^{n}(x_{i}-c_{i})Q_{i}$. Specifically, $Q(x_1,\ldots, x_n)$ has the form
\begin{equation*}\left\{ \begin{pmatrix} \mathbf{y_{1}}\\ \vdots \\ \mathbf{y_{n}} \end{pmatrix}\left| \begin{pmatrix} \mathbf{C_{1}} &  \cdots & \mathbf{0} \\   \vdots &  \ddots & \vdots \\ \mathbf{0}  & \cdots & \mathbf{C_{n}} \end{pmatrix}
\begin{pmatrix} \mathbf{y_{1}}\\ \vdots \\ \mathbf{y_{n}} \end{pmatrix} =
x_{1}\begin{pmatrix} \mathbf{d_{1}}\\\mathbf{0} \\ \vdots \\ \mathbf{0} \end{pmatrix}+ \cdots + x_{n}\begin{pmatrix} \mathbf{0}\\ \vdots \\ \mathbf{0}  \\ \mathbf{d_{n}} \end{pmatrix} + \mathbf{e}, \mathbf{y} \geq \mathbf{0}\right\},\right.
\end{equation*}
where the $i$-th entry of $\mathbf{e}$ is equal to $-c_i\mathbf{d_{i}}$.
\end{proof}

This corollary allows us to obtain monomials as weight functions.

\begin{cor}\label{monomial}
For every monomial $w(\mathbf{x}) = \mathbf{x}^{\alpha} = x_{1}^{\alpha_{1}}x_{2}^{\alpha_{2}}\cdots x_{n}^{\alpha_{n}}$, there exists a weight lifting polytope $P^{*}\subseteq \mathbb{R}^{n^*}$ where $n^* = n + 2|\alpha| = n + 2\sum_{i=1}^n \alpha_i$ such that
$$  \sum_{\mathbf{x} \in P \cap \Z^{n}}w( \mathbf{x} ) =  \left| P^{*} \cap \Z^{n^{*}} \right|.$$
\end{cor}
\begin{proof}
By Corollary \ref{lifting}, we just need to show that $x_{i}^{\alpha_{i}}$ is a {late-dilated} Ehrhart polynomial. It is well known that $(t+1)^{\alpha_{i}}$ is the Ehrhart polynomial of the $\alpha_{i}$-dimensional hypercube. In particular, the $t$-th dilation of the hypercube has the form
\begin{equation*}
t Q_{i} = \left\{ \begin{pmatrix} y_{1} \\ \vdots \\ y_{\alpha_{i}} \\ z_{1} \\ \vdots \\ z_{\alpha_{i}} \end{pmatrix} \left| \begin{bmatrix} \begin{matrix} 1 & 0 & \cdots & 0 \\ 0 & 1 & \cdots & 0 \\  \vdots & \vdots & \ddots & \vdots \\ 0 & 0 & \cdots & 1  \end{matrix} & \begin{matrix} 1 & 0 & \cdots & 0 \\ 0 & 1 & \cdots & 0 \\  \vdots & \vdots & \ddots & \vdots \\ 0 & 0 & \cdots & 1  \end{matrix} \end{bmatrix} \begin{pmatrix} y_{1} \\ \vdots \\ y_{\alpha_{i}} \\ z_{1} \\ \vdots \\ z_{\alpha_{i}} \end{pmatrix} = t\begin{pmatrix} 1 \\ \vdots \\ 1  \end{pmatrix} , y_{i}\geq 0,z_{i}\geq 0 \right. \right\}.
\qedhere
\end{equation*}
\end{proof}

In the next corollary we consider weight functions that are polynomials.

\begin{cor}\label{polynomial}
For every polynomial $w( \mathbf{x} ) = \sum_{\alpha \in I}c_{\alpha} \mathbf{x} ^{\alpha} = \sum_{\alpha \in I}c_{\alpha}x_{1}^{\alpha_{1}}x_{2}^{\alpha_{2}}\cdots x_{n}^{\alpha_{n}}$, there exist $|I|$ weight lifting polytopes $P^{*}_{\alpha}$ indexed by the exponents of monomials such that
$$  \sum_{ \mathbf{x} \in P \cap \Z^{n}}w(x) =  \sum_{\alpha \in I}c_{\alpha}\left| P^{*}_{\alpha} \cap \Z^{n^{*}} \right|.$$
\end{cor}
\begin{proof}
This follows directly from Corollary \ref{monomial}
\end{proof}

\begin{rem}
Corollary \ref{polynomial} implies that if $w(x)$ is a polynomial with $|I|$ nonzero monomials, then we can compute the sum of lattice points of $P$ weighted by $w$ by counting integral points in $|I|$ weight lifting polytopes.
\end{rem} 

We give another two corollaries of Theorem \ref{polytopelifting}. 
\begin{cor}\label{choose}
Consider the polynomial $w( \mathbf{x} ) = \prod_{i=1}^{n} \binom{x_{i}+\alpha_{i}-1}{\alpha_{i} -1 }$. There exists a weight lifting polytope $P^{*} \subseteq \mathbb{R}^{n^*}$ where $n^* = n+|\alpha|$ such that 
$$  \sum_{ \mathbf{x} \in P \cap \Z^{n}}w( \mathbf{x} ) =  \left| P^{*} \cap \Z^{n^{*}} \right|.$$
\end{cor}
\begin{proof}
Recall that $\binom{x_{i}+\alpha_{i}-1}{\alpha_{i} -1 }$ is the Ehrhart polynomial of the standard $(\alpha_{i}-1)$-simplex $1 = y_{1} + \cdots + y_{\alpha_{i}}$ with $y_{i} \geq 0$.
In particular, the $t$-th dilation of the simplex has the form
$$t Q_{i} = \left\{ \begin{pmatrix} y_{1} \\ \vdots \\ y_{\alpha_{i}}  \end{pmatrix} \left| \begin{bmatrix}  1  & \cdots & 1  \end{bmatrix} \begin{pmatrix} y_{1} \\ \vdots \\ y_{\alpha_{i}} \end{pmatrix} = t\cdot 1 , y_{i}\geq 0 \right. \right\}.$$
Applying Corollary \ref{lifting} gives the weight lifting polytope from the statement.
\end{proof}

In Corollary \ref{monomial} we express the weighted sum by a monomial of the lattice points of $P$ using a single $P^* \subset \mathbb{R}^{n + 2|\alpha|}$.
In Corollary \ref{monomial2} below, we present a complementary point of view. Namely, we express this sum using at most $(\alpha_{1}+1)(\alpha_{2}+1)\cdots (\alpha_{n}+1)$ polytopes of lower dimension $P^*_{\beta} \subset \mathbb{R}^{n + |\beta|}$.

\begin{cor}\label{multichoose}
Consider the polynomial $w( \mathbf{x} ) = \prod_{i=1}^{n} \binom{x_{i}}{\alpha_{i} -1 }$. There exists a weight lifting polytope $P^{*}$ of the dimension $n^{*} = n+|\alpha|$ such that
  $\sum_{ \mathbf{x} \in P \cap \Z^{n}}w( \mathbf{x} ) =  \left| P^{*} \cap \Z^{n^{*}} \right|.$
\end{cor}
\begin{proof}
The function $\binom{x_{i}}{\alpha_{i} -1 }$ is a {late-dilated} Ehrhart polynomial because $\binom{x_{i} + \alpha_{i} -1 }{ \alpha_{i} -1 }$ is the Ehrhart polynomial of the standard $(\alpha_i-1)$-simplex. Applying Corollary \ref{lifting} gives the desired weight lifting polytope.
\end{proof}

Note that $\{ \binom{x+k-1}{k -1 } \mid k=1,2,\dots \}$ and $\{\binom{x}{k -1 } \mid k=1,2,\dots \}$ are two well-known bases of the vector space of polynomials in $x$.

\begin{cor}\label{monomial2}
For every monomial $w( \mathbf{x} ) =  \mathbf{x}^{\alpha} = x_{1}^{\alpha_{1}}\cdots x_{n}^{\alpha_{n}}$, there exist at most $(\alpha_{1}+1)\cdots (\alpha_{n}+1)$ weight lifting polytopes $P^{*}_{\beta}$ indexed by the vector $\beta$ and $P^{*}_{\beta} \subset \mathbb{R}^{n^{*}}$ where $n^{*} = n + |\beta|$
such that
$$  \sum_{ \mathbf{x} \in P \cap \Z^{n}}w( \mathbf{x} ) =  \sum_{\beta \leq \alpha}c_{\beta} \left| P^{*}_{\beta} \cap \Z^{n^{*}} \right|.$$
\end{cor}
\begin{proof}
Let $v_{k}(x)$ be one of the two binomial bases described above.
We can transform the monomial basis $\{x^k\mid k = 0,1,2,\ldots\}$
into the binomial basis, 
\[
x_{1}^{\alpha_{1}}x_{2}^{\alpha_{2}}\cdots x_{n}^{\alpha_{n}} = \sum_{\beta \leq \alpha}c(\alpha,\beta) \cdot v_{\beta_{1}}(x_{1})v_{\beta_{2}}(x_{2})\cdots v_{\beta_{n}}(x_{n}).
\]
By Corollary \ref{choose} and Corollary \ref{multichoose}, for each $\beta$ and each polynomial $v_{\beta_{1}}(x_{1})v_{\beta_{2}}(x_{2})\cdots v_{\beta_{n}}(x_{n})$, there exists a corresponding weight lifting polytope $P^{*}_{\beta} \subset \mathbb{R}^{n + |\beta|}$.
\end{proof}

We now give the proof of \Cref{integrate-maximize}. 
Recall from its statement that we are dealing with the most general \emph{Ehrhart quasi-polynomial weighted} case, i.e., the weight $w(x)$ is a non-constant quasi-polynomial counting lattice points of a parametric polyhedron, as in the statement of \Cref{polytopelifting}.

\begin{proof}[Proof of Corollary \ref{integrate-maximize}]
Applying Theorem \ref{polytopelifting} to $P$ gives a weight lifting polytope $P^*$ for which $\weightsum{P}{w}{t} = \weightsum{P^*}{1}{t}$.  Applying a classical result relating the volume and lead coefficient of the Ehrhart quasi-polynomial of $P^*$ completes the proof.  Both $\weightsum{P}{w}{t}$ and $\weightsum{P^*}{1}{t}$ are quasi-polynomial functions of $t$, and concretely, this equality implies that their leading coefficients are the same.

We can then replace integration of $w(x)$ over $P$ with computation of the leading coefficient of $\weightsum{P^*}{1}{t}$, which is equivalent to computing the volume of $P^*$. 
Note that this transformation can be carried out in a number of steps that is polynomial in the size of the inputs describing $P^*$.

For the second claim, we start by recalling an elementary fact. Let $S = \{s_1,\ldots, s_r\}$ be a set of non-negative real numbers.  Then $\max\{s_i \mid s_i \in S\} = \lim_{k \rightarrow \infty}\sqrt[k]{\sum_{j=1}^r s_j^k}$.
The arithmetic mean of $S$ is at most its maximum value, which in turn is at most as big as $\sum_i s_i$. We apply these ideas to the set $S = \{w(\alpha) \mid \alpha \in P \cap \Z^n\}$. This gives upper and lower bounds for each positive integer $k$:
\[
L_{k}=
\sqrt[k]{\frac{\sum\limits_{\alpha\in P \cap \Z^n}
w(\alpha)^{k}}{|P\cap\Z^d|}} \leq
\max\{w(\alpha)\mid \alpha\in P\cap\Z^n\}\leq
\sqrt[k]{\sum\limits_{\alpha\in P \cap \Z^n}
w(\alpha)^{k}}=U_{k}.
\]

As $k\rightarrow\infty$, $L_{k}$ and $U_{k}$ approach
this maximum value monotonically (from below and above, respectively).
Trivially, if the difference between the (rounded) upper and lower bounds
becomes strictly less than $1$, we have determined $\max\{w(x)\mid x\in P\cap\Z^n\}=\lceil L_{k}\rceil$. Thus the process 
terminates with the correct value. Finally, the key value in the sequences $L_{k}$ and $U_{k}$  is the term $\weightsum{P}{w^k}{t}=\sum\limits_{\alpha\in tP \cap \Z^n} w(\alpha)^{k}$. Corollary \ref{lifting} describes how to construct the weight lifting polytope $P^*$ corresponding to the pair $P$ and $w(\alpha)^k$.
\end{proof}

\section{Applications} \label{applications}

Theorem \ref{polytopelifting} has applications beyond integration and maximization of Ehrhart quasi-polynomials.  In this section we discuss how to use it to find new algebraic combinatorial identities by carefully choosing the polytope $P$ and reinterpreting the weight function $w$ in terms of Ehrhart quasi-polynomials of some polytopes $Q_i$.

\subsection{Weighted Ehrhart in number theory}\label{WENT}

\subsubsection{Simultaneous core partitions.}
We first describe an area in which weighted Ehrhart machinery has already been applied to prove a significant result.  Let $\lambda$ be a partition and $\mathcal{H}(\lambda)$ denote its multiset of hook lengths.  The partition $\lambda$ is called an \emph{$a$-core partition} if no element of $\mathcal{H}(\lambda)$ is divisible by $a$.  If $\lambda$ is both an $a$-core partition and a $b$-core partition, then we say that it is an $(a,b)$-core partition.  There is an extensive literature about statistical properties of sizes of simultaneous core partitions \cite{CoreSurveyNam,CoreSurveyNath}.  Anderson gave a formula for the number of simultaneous $(a,b)$-core paritions in the case where $a$ and $b$ are relatively prime.
\begin{theorem}\cite[Theorem 1]{Anderson}\label{AndersonThm}
Suppose $a$ and $b$ are relatively prime.  The number of simultaneous $(a,b)$-core partitions is 
\[
\frac{1}{a+b}\binom{a+b}{a}.
\]
\end{theorem}
\noindent Johnson proved a conjecture of Armstrong, computing the average size of this finite set of partitions.
\begin{theorem}\cite[Theorem 7]{Johnson}\label{JohnsonThm}
Suppose $a$ and $b$ are relatively prime.  The average size of an $(a,b)$-core partition is $(a+b+1)(a-1)(b-1)/24$.
\end{theorem}

Johnson's proof of Theorem \ref{JohnsonThm} fits into the framework of weighted Ehrhart theory.  Suppose that $a$ and $b$ are relatively prime positive integers.  It is not hard to show that $a$-core partitions are in bijection with elements of $\Lambda_a = \left\{ (c_0,\ldots, c_{a-1}) \in \Z^a \colon \sum_i c_i = 0\right\}$. Let $r_a(x)$ be the remainder when $x$ is divided by $a$.  We use cyclic indexing for elements $\mathbf{c} \in \Lambda_a$, that is, for $k \in \Z$ we set $c_k = c_{r_a(k)}$.  Simultaneous $(a,b)$-core partitions are in bijection with the elements of $\Lambda_a$ satisfying the inequalities $c_{i+b}-c_i \le  \left\lfloor \frac{b+i}{a} \right\rfloor$ for each $i \in \{0,1,\ldots, a-1\}$ \cite[Lemma 23]{Johnson}.  In this way, we see that $(a,b)$-core partitions are in bijection with integer points in a rational polytope $\mathrm{SC}_a(b)$.  The size of the $a$-core partition corresponding to $\mathbf{c} = (c_0,\ldots,c_{a-1})$ is $h_a(\mathbf{c}) = \frac{a}{2} \sum_{i=0}^{a-1} \left(c_i^2 + i c_i\right)$ \cite[Theorem 22]{Johnson}.  Therefore, we have the following interpretations of the theorems above.
\begin{enumerate}
\item Theorem \ref{AndersonThm} is equivalent to counting the number of integer points in $\mathrm{SC}_a(b)$.
\item Theorem \ref{JohnsonThm} is equivalent to computing $\sum_{\mathbf{c} \in \mathrm{SC}_a(b)} h_a(\mathbf{c})$.
\end{enumerate}

Johnson computes this weighted sum of lattice points by relating it
to a sum over the subset of integer points $(z_0,\ldots, z_{a-1})$ of the dilation of the standard simplex $b\Delta_{a-1}$ that satisfy $\sum i z_i \equiv 0 \pmod{a}$.  Johnson then shows that the sum he needs to compute is equal to $1/a$ times the sum of a quadratic function $w$ taken over all integer points of $b\Delta_{a-1}$. In order to conclude, he applies a result from Euler--Maclaurin theory, which is a version of the first part of Corollary \ref{integrate-maximize}, and also applies a version of weighted Ehrhart reciprocity that appears in \cite{FedericoErwan}. 

By Corollary \ref{polynomial}, there exists a family of weight lifting polytopes $P^{*}_{\alpha} \subset \mathbb{R}^{n^{*}}$ such that
\[
\sum_{x\in b\triangle_{a-1} \cap \mathbb{Z}^{a} }w(x) = \sum_{\alpha \in I}c_{\alpha}\left| P^{*}_{\alpha} \cap \Z^{n^{*}} \right|.
\]
It seems likely that further study of these kinds of weight lifting polytopes can lead to new techniques in the study of simultaneous core partitions.

\subsubsection{Numerical semigroups.}
A numerical semigroup $S$ is an additive submonoid of $\mathbb{N}_0 = \{0,1,2,\ldots\}$ with finite complement.  The elements of $\mathbb{N}_0 \setminus S$ are the \emph{gaps} of $S$, denoted $G(S) = \{h_{1},\ldots ,h_{g}\}$.  The number of gaps of $S$ is called its \emph{genus}, $g(S) = |G(S)| = g$. 
The \emph{weight} of $S$ is defined by $w(S) = (h_{1}+ \cdots+h_{g})-(1+2+\cdots+g)$.  The motivation for studying $w(S)$ comes from the theory of Weierstrass semigroups of algebraic curves \cite[Chapter 1, Appendix E]{ACGH}.

Numerical semigroups containing $m$ are in bijection with nonnegative integer points $(x_1,\ldots, x_{m-1}) \in \Z_{\ge 0}^{m-1}$ in the \emph{Kunz polyhedron} $P_{m}\subset \R^{m-1}$, which is defined via bounding inequalities
\[
x_{i} + x_{j} \geq x_{i+j} \text{ if } i+j < m,\quad
    x_{i} + x_{j} + 1 \geq x_{i+j - m} \text{ if }i+j > m.
\]
Let $NS(m,g)$ be the set of numerical semigroups containing  $m$ with genus $g$.  These semigroups are in bijection with the integer points of $P_{m,g}$, the polytope we get from 
$P_m$ by adding the additional constraint $\sum x_i = g$.  For a more extensive discussion of the connection between numerical semigroups containing $m$ and integer points in the Kunz polyhedron, see \cite[Section 4]{KaplanONeill}.
If $(k_1,\dots, k_{m-1})$ corresponds to a semigroup $S$, then 
\[
w(S) = \frac{m}{2}\sum_{i=1}^{m-1}k_i(k_i-1) + \sum_{i=1}^{m-1}i k_i - \frac{1}{2}\left(\sum_{i=1}^{m-1} k_i \right)\left(1+\sum_{i=1}^{m-1} k_i \right).
\]
There has been recent interest in statistical properties of weights of semigroups, see \cite[Section 5]{KaplanYe} and~\cite{KaplanSinghal}.

By Corollary \ref{polynomial}, there exists a family of weight lifting polytopes $P^{*}_{\alpha}\subset \mathbb{R}^{n^*}$ such that
\[
\sum_{S \in NS(g,m)}w(S) = \sum_{S \in P_{m,g}\cap \mathbb{Z}^{m-1}}w(S) = \sum_{\alpha \in I}c_{\alpha}\left| P^{*}_{\alpha} \cap \Z^{n^{*}} \right|.
\]

Studying this family of polytopes and applying a version of Corollary \ref{integrate-maximize} suggests an approach to the following question.
\begin{pbm}
For fixed $m$, what is the main term in the expression for $\sum_{S \in NS(g,m)}w(S)$ as $g\rightarrow \infty$?
\end{pbm}
\noindent In Section \ref{experiments}, we present some computational data that was generated using the ideas described in this section.

It is not so difficult to show that for a numerical semigroup $S$ with $g(S) = g$, $w(S) \le g(g-1)/2$, and that equality holds if and only if $2 \in S$.  The smallest nonzero element of a numerical semigroup $S$ is called its multiplicity and is denoted by $m(S)$.
\begin{pbm}
For fixed $m$, what is the maximum of $w(S)$ taken over the finite set of numerical semigroups with $m(S) = m$ and $g(S) = g$?
\end{pbm}
It is not difficult to work out the answer to this question for $m=3$, but in general not much is known. We suspect that for fixed $m$ and increasing $g$ the answer should grow like a constant depending on $m$ times $g^2$.  The second part of Corollary \ref{integrate-maximize} suggests a promising approach to this problem.

\subsection{Weighted Ehrhart in combinatorial representation theory}

There is a long tradition of using lattice points of polytopes in representation theory (see \cite{GTDeLoera+McAllister2004} and the references there). For example, it is well known that Gelfand-Tsetlin polytopes have important connections to the representation theory of $\mathfrak{gl}_n(\C)$.
Here, as an application of Theorem \ref{polytopelifting}, we provide new connections.

\subsubsection{Maximizing Kostka numbers.}
Fix a partition $\lambda \vdash n$ and let $SSYT(\lambda)$ denote the set of semi-standard Young tableaux of shape $\lambda$.
The Schur function $s_\lambda$ is 
$$s_{\lambda}(x) = \sum_{T \in SSYT(\lambda)} x^{T}= \sum_{\alpha \in \text{comp}(n)} K_{\lambda\alpha}x^{\alpha},$$
where $\text{comp}(n)$ is the set of weak compositions $n$ and $K_{\lambda\alpha}$ is the \emph{Kostka number} that counts the number of tableaux in $SSYT(\lambda)$ with content $\alpha$.
Evaluating $s_\lambda$ at $x_{1}=1,x_{2} = 1,\ldots,x_{N} = 1,x_{N+1} = 0,x_{N+2} = 0,\ldots$ yields 
$$ |SSYT(\lambda,N)| = \sum_{\alpha \in N\text{-comp}(n)} K_{\lambda\alpha},$$
where $SSYT(\lambda,N)$ is the set of semi-standard Young tableaux of shape $\lambda$ and entries bounded by $N$ and $N\text{-comp(n)}$ is the set of weak composition of $n$ with $N$ parts.

A weak composition of $n$ with $N$ parts is a lattice point in the scaled standard $(N-1)$-simplex $n\triangle_{N-1}$. The Kostka number $K_{\lambda\alpha}$ equals the number of lattice points in the Gelfand--Tsetlin polytope $GT(\lambda,\alpha)$ (see e.g.,
\cite{GTDeLoera+McAllister2004}), so $w(\alpha)=K_{\lambda\alpha}$ is a weight function. 
There have been contributions to understanding the behavior of $K_{\lambda\alpha}$ as $(\lambda,\alpha)$ vary and an example is \cite{HuhMatherneMeszarosStDizier} in which it is shown that they are log-concave.
Applying the method in Corollary \ref{integrate-maximize} one can use the weight lifting polytope given by Theorem \ref{polytopelifting} to compute $max_{\alpha \in N\text{-comp(n)}} K_{\lambda\alpha}$.

\subsubsection{Robinson-Schensted-Knuth (RSK) identity}
Fix partitions $\mu ,\nu \vdash n$ and recall the famous \emph{RSK identity} (for an introduction and 
details see e.g., \cite{combinatorialreptheory}):
$$ \sum_{\lambda \vdash n }K_{\lambda\mu}K_{\lambda\nu} =   N_{\mu,\nu}.$$
The left sum is over partitions of $n$ and the summands are products of  {Kostka numbers}. 
In fact, the left side of the identity is a weighted sum over the lattice points of 
\begin{equation}\label{eq_partition_simplex}
    P=\{\mathbf{x}\in\R^n \mid x_1+\cdots+x_n=n, x_1\ge x_2\ge \cdots \ge x_n\ge 0 \}.
\end{equation}
This is because the weight function  $w(\lambda)=K_{\lambda\mu}K_{\lambda\nu}$ is the number of lattice points in the Cartesian product $GT(\lambda,\mu)\times GT(\lambda,\nu)$ of two {Gelfand--Tsetlin polytopes}. 
The right-hand side of RSK, $N_{\mu\nu}$,
is the number of lattice points in the transportation polytope 
$$Mat_{n,n}(\mu,\nu) = \left\{(z_{ij})_{1 \leq i,j \leq n}\mid \sum_{j}z_{ij} = \mu_{i},\sum_{i}z_{ij} = \nu_{j}, z_{ij} \geq 0\right\}.$$

While RSK provides more information (e.g., a bijection), Theorem \ref{polytopelifting} gives a new polytope whose number of lattice points is the sum $\sum_{\lambda \vdash n }K_{\lambda\mu}K_{\lambda\nu}$.

\begin{cor}[A new RSK-like identity]
There exists a weight lifting 
polytope $P^{*}(\mu,\nu)\subseteq\mathbb{R}^{n^{2}+2n}$ which is combinatorially different from $Mat_{n,n}(\mu,\nu)$ such that
$$ \sum_{\lambda \vdash n }K_{\lambda\mu}K_{\lambda\nu} =   |P^{*}(\mu,\nu)\cap \mathbb{Z}^{n^{2}+2n}|.$$
\end{cor}

\subsubsection{Littlewood--Richardson Coefficients.}
Schur functions are central objects in representation theory and combinatorics.
The skew Schur function for partitions $\lambda,\mu \vdash n$ is
$$s_{\lambda / \mu}(x) = \sum_{\alpha \in \text{comp}(n)} K_{\lambda/\nu,\alpha}x^{\alpha},$$
where the sum is over all compositions of $n$ and $K_{\lambda/\nu,\alpha}$ counts the number of skew semi-standard Young tableaux of shape $\lambda/\nu$ and weight $\alpha$.
The Littlewood--Richardson rule (see e.g., \cite{Sagan}) expresses the skew Schur functions in terms of Schur functions,
$$ s_{\lambda / \mu}(x) =  \sum_{\nu \vdash n}c_{\mu\nu}^{\lambda}s_{\nu}(x).$$
Comparing the expression of the coefficient of the monomial $x^{\alpha}$ yields
$$K_{\lambda/\nu,\alpha} = \sum_{\nu \vdash n}c_{\mu\nu}^{\lambda}K_{\nu\alpha}.$$
The Littlewood--Richardson coefficient $c_{\mu\nu}^{\lambda}$ counts the number of lattice points in the \emph{hive polytope} $H_{\mu\nu}^{\lambda}$ (see e.g., \cite{hive_polytopes}).
Applying Theorem \ref{polytopelifting} to the simplex in \eqref{eq_partition_simplex} and the weight function
$w(\nu)=c_{\mu\nu}^{\lambda}K_{\nu\alpha}$, which counts the number of lattice points in $H_{\mu\nu}^{\lambda}\times GT(\nu,\alpha)$, we obtain the following corollary.

\begin{cor}
There exists a weight lifting polytope $P^{*}(\lambda/\mu,\alpha)\subseteq\mathbb{R}^{n^{2}+2n}$ such that
$$ \sum_{\lambda \vdash n }c_{\mu\nu}^{\lambda}K_{\nu\alpha} =   |P^{*}(\lambda/\mu,\alpha)\cap \mathbb{Z}^{n^{2}+2n}|.$$
\end{cor}

\subsubsection{Newell-Littlewood Polytopes.}
The Newell-Littlewood coefficients are defined from Littlewood--Richardson coefficients by the cubic expression,
$$ N_{\mu,\nu,\lambda,} = \sum_{\alpha,\beta,\gamma}c_{\alpha,\beta}^{\mu}c_{\alpha,\gamma}^{\nu}c_{\beta,\gamma}^{\lambda},$$
where the summation is taken over all possible integer  partitions with $|\alpha|+|\beta| = |\mu|$, $|\alpha|+|\gamma| = |\nu|$ and $|\beta|+|\gamma| = |\lambda|$. By Theorem~\ref{polytopelifting}, there exists a weight lifting polytope $P^{*}(\mu,\nu,\lambda)\subseteq  \mathbb{Z}^{n^{*}}$ such that
$$ \sum_{\alpha,\beta,\gamma}c_{\alpha,\beta}^{\mu}c_{\alpha,\gamma}^{\nu}c_{\beta,\gamma}^{\lambda} = |P^{*}(\mu,\nu,\lambda) \cap \mathbb{Z}^{n^{*}}|.$$

Note that for partitions $\mu,\nu,\lambda$ with at most $k$ parts, Gao, Orelowitz, and Yong \cite[Section 5]{Gaoetal2021} define \emph{Newell-Littlewood polytopes} $\mathcal{P}_{\mu,\nu,\lambda}\subseteq \R^{3n^2}
$ with $N_{\mu,\nu,\lambda,} = |\mathcal{P}_{\mu,\nu,\lambda}\cap  \mathbb{Z}^{3k^{2}}|$. It would be interesting to compare these polytopes 
with our weight lifting polytopes.

\section{Experiments} \label{experiments}

    \subsection{Integration over polytopes}
        We present an experiment related to symbolic computing. Let $P$ be a $d$-dimensional rational convex polyhedron inside $\R^n$ and let $w \in \Q[x_1,\dots,x_n]$ be a (homogeneous) polynomial with rational coefficients. We consider the problem of efficiently computing the \emph{exact} value of the integral of the polynomial $w$ over $P$, denoted $\int_{P} w \dif m$, where $\dif m$ is the \emph{integral Lebesgue measure} on the affine hull of the polytope $P$. For rational inputs, the output will always be a rational number $\int_P f \dif m$. Integration over polytopes was studied extensively in \cite{Barvinok-1991}, \cite{Barvinok-1992} and more recently in \cite{BBDKV08,DDM2013}.

        Integration over polytopes is in general an important but difficult problem. (See \cite{BBDKV08,Gritzmann+Klee} and the references therein). Our contribution starts from an old observation: It is known that the computation of the leading coefficient of $\weightsum{P}{w}{t}$ is the same as computing the integral of $w$ over the polytope $P$ (see \cite{BBDKV08}). Theorem \ref{polytopelifting} provides a new avenue to compute that leading coefficient because $\weightsum{P}{w}{t}=\weightsum{P^*}{1}{t}$.  We can replace integration over $P$ with computation of the leading coefficient of a usual Ehrhart leading coefficient, which is a direct volume computation. 

        In the paper \cite{DDM2013} the authors released integration software implemented in LattE\footnote{Code is available from \url{https://www.math.ucdavis.edu/~latte/}.} that includes several fast algorithms. They depend on two main facts. The first is that integrals of arbitrary \emph{powers of  linear forms} can be computed in polynomial time. Waring's theorem for polynomials says that any polynomial can be written as a finite sum of \emph{powers of linear forms}. 
        Therefore, to integrate an input polynomial LattE uses such a  decomposition, $\sum_{\ell}c_{\ell} \langle \ell, x \rangle^M$. The second fact, is that we can triangulate any polytope and just do the integral over simplices, then add the pieces. For all details of the LattE Integration algorithms and implementation see \cite{alexanderhirschowitz,Barvinok-1991,Barvinok-1992,BBDKV08,DDM2013}.

         For this paper we implemented our new algorithm, which we the \emph{WLPvolume} algorithm, in SAGE. We apply Theorem \ref{polytopelifting} to find the weight lifting polytope and then we compute its volume using existing LattE code. We compared this with LattE \emph{Integration} as implemented in \cite{DDM2013}, which is available in the latest LattE release. Table \ref{tab:monomialsimplex} and Table \ref{tab:linearformsimplex}  present a comparison of the LattE Integration method and the WLPvolume method. The columns give the dimension of the polytope and the rows give the degree of the integrand. Each cell has two running times, the LattE Integration method is the top one and the WLPvolume method is the bottom one. 

        Table \ref{tab:monomialsimplex} provides data where we integrate a monomial over the standard simplex. The LattE Integration method is extremely slow when both dimension and degree are high. LattE's algorithm needs to turn one monomial into a sum of powers of linear forms.  This decomposition often involves thousands of linear forms. In the final case included in the table, the algorithm did not even finish. In contrast WLPvolume completed this computation in a little over 3 seconds. On the other hand, Table \ref{tab:linearformsimplex} provides data when we integrate a power of a linear form over the standard simplex. The performance reverses and the WLPvolume method is extremely inefficient when both dimension and degree are high. The WLPvolume algorithm often needs to decompose the constructed weight lifting polytope into thousands of simplicial cones to compute the volume.

        \begin{table}[htbp]
            \centering
            \begin{tabular}{ m{2em} || c c c c c c c c c c }
                 & \multicolumn{10}{c}{Dimension of the simplex} \\
                Deg & 1 & 2 & 3 & 4 & 5 & 6 & 7 & 8 & 9 & 10 \\
                \hline\hline
                \multirow{2}{0.8em}{1}
                 & 0.01 & 0.00 & 0.00 & 0.01 & 0.01 & 0.02 & 0.01 & 0.02 & 0.02 & 0.06 \\
                 & 0.00 & 0.01 & 0.01 & 0.02 & 0.02 & 0.02 & 0.03 & 0.05 & 0.06 & 0.08 \\
                \hline
                \multirow{2}{0.8em}{2}
                & 0.00 & 0.01 & 0.01 & 0.01 & 0.01 & 0.04 & 0.12 & 0.41 & 1.43 & 5.13 \\
                & 0.01 & 0.01 & 0.01 & 0.03 & 0.04 & 0.06 & 0.10 & 0.13 & 0.21 & 0.31 \\
                \hline
                \multirow{2}{0.8em}{3}
                & 0.01 & 0.00 & 0.00 & 0.02 & 0.04 & 0.21 & 1.01 & 5.12 & 25.87 & 123.82 \\
                & 0.00 & 0.01 & 0.03 & 0.03 & 0.09 & 0.13 & 0.21 & 0.35 & 0.55 & 0.80 \\
                \hline
                \multirow{2}{0.8em}{4}
                & 0.00 & 0.00 & 0.01 & 0.02 & 0.11 & 0.80 & 5.38 & 35.49 & 222.07 & 1345.95 \\
                & 0.01 & 0.01 & 0.03 & 0.07 & 0.13 & 0.27 & 0.45 & 0.74 & 1.18 & 1.73 \\
                \hline
                \multirow{2}{0.8em}{5}
                & 0.00 & 0.01 & 0.01 & 0.04 & 0.29 & 2.55 & 21.53 & 166.20 & 1282.92 & - \\
                & 0.01 & 0.02 & 0.06 & 0.12 & 0.25 & 0.48 & 0.87 & 1.39 & 2.21 & 3.32 \\
                \hline
            \end{tabular}
            \caption{Integration(upper) v.s. WLPvolume(lower): one monomial $\prod x_i^{row}$  over the standard $col$-simplex.}
            \label{tab:monomialsimplex}
        \end{table}

        \begin{table}[htbp]
            \centering
            \begin{tabular}{ m{2em} || c c c c c c c c c c }
                 & \multicolumn{10}{c}{Dimension of the simplex} \\
                Deg & 1 & 2 & 3 & 4 & 5 & 6 & 7 & 8 & 9 & 10 \\
                \hline\hline
                \multirow{2}{0.8em}{2}
                & 0.01 & 0.01 & 0.01 & 0.01 & 0.01 & 0.01 & 0.01 & 0.01 & 0.01 & 0.02 \\ 
                & 0.01 & 0.01 & 0.01 & 0.02 & 0.01 & 0.01 & 0.02 & 0.01 & 0.03 & 0.03 \\
                \hline
                \multirow{2}{0.8em}{3}
                & 0.01 & 0.00 & 0.01 & 0.00 & 0.01 & 0.02 & 0.01 & 0.01 & 0.01 & 0.01 \\ 
                & 0.01 & 0.01 & 0.02 & 0.02 & 0.01 & 0.02 & 0.03 & 0.03 & 0.04 & 0.05 \\
                \hline
                \multirow{2}{0.8em}{4}
                & 0.00 & 0.01 & 0.01 & 0.01 & 0.01 & 0.00 & 0.00 & 0.01 & 0.02 & 0.02 \\
                & 0.01 & 0.01 & 0.01 & 0.02 & 0.02 & 0.03 & 0.03 & 0.07 & 0.11 & 0.15 \\
                \hline
                \multirow{2}{0.8em}{5}
                & 0.00 & 0.01 & 0.01 & 0.00 & 0.02 & 0.01 & 0.01 & 0.01 & 0.01 & 0.01 \\
                & 0.01 & 0.01 & 0.02 & 0.02 & 0.03 & 0.05 & 0.10 & 0.17 & 0.32 & 0.54 \\
                \hline
                \multirow{2}{0.8em}{6}
                & 0.00 & 0.01 & 0.00 & 0.01 & 0.00 & 0.01 & 0.01 & 0.01 & 0.01 & 0.01 \\ 
                & 0.00 & 0.01 & 0.02 & 0.03 & 0.04 & 0.08 & 0.21 & 0.41 & 1.48 & 3.43 \\
                \hline
                \multirow{2}{0.8em}{7}
                & 0.01 & 0.00 & 0.00 & 0.00 & 0.02 & 0.02 & 0.01 & 0.01 & 0.01 & 0.02 \\ 
                & 0.01 & 0.02 & 0.02 & 0.03 & 0.06 & 0.19 & 0.53 & 1.61 & 6.42 & 14.95 \\
                \hline
                \multirow{2}{0.8em}{8}
                & 0.00 & 0.01 & 0.01 & 0.01 & 0.00 & 0.01 & 0.01 & 0.01 & 0.01 & 0.01 \\
                & 0.01 & 0.02 & 0.03 & 0.04 & 0.10 & 0.33 & 1.13 & 5.49 & 20.54 & 72.00 \\
                \hline
                \multirow{2}{0.8em}{9}
                & 0.00 & 0.00 & 0.01 & 0.01 & 0.01 & 0.01 & 0.01 & 0.01 & 0.02 & 0.01 \\
                & 0.02 & 0.03 & 0.03 & 0.04 & 0.15 & 0.76 & 2.85 & 16.05 & 76.25 & 236.78 \\
                \hline
                \multirow{2}{0.8em}{10}
                & 0.00 & 0.01 & 0.01 & 0.01 & 0.01 & 0.00 & 0.01 & 0.01 & 0.02 & 0.01 \\
                & 0.02 & 0.02 & 0.02 & 0.06 & 0.30 & 1.43 & 6.68 & 38.55 & 231.76 & 1694.71 \\
                \hline
            \end{tabular}
            \caption{Integration(upper) v.s. WLPvolume(lower): a power of a linear form $(\sum c_ix_i)^{row}$ over the standard $col$-simplex.}
            \label{tab:linearformsimplex}
        \end{table}

    \subsection{Weights of numerical semigroups}
        As we described in Section \ref{WENT}, numerical semigroups containing an integer $m$ with genus $g$ are in bijection with the lattice points inside the Kunz polytope $P_{m,g}$ and the weight of a numerical semigroup $S$ is a quadratic polynomial in the coordinates of the corresponding point. Therefore, we can study the average weight of these finitely many numerical semigroups for each $m$ and $g$.

        Note that the weight of a numerical semigroup with genus $g$ is the sum of elements in the gaps minus the sum from $1$ to $g$ and the element in the gaps are distinct $g$ integers ranging from $1$ to $2g-1$. So roughly speaking, the weight of a numerical semigroup has a quadratic growth with respect to the genus $g$. We expect that for fixed $m \ge 2$, as $g \rightarrow{\infty}$, the average weight of a semigroup containing $m$ with genus $g$ should grow like a constant depending on $m$ times $g^2$.  We implemented the weight lifting method using algorithms in LattE to collect data for $3 \leq m \leq 8$ and $g \leq 200$.

        The results are presented in Figure \ref{fig:averageplot}. For each $m$ and $g$, we first calculate the sum of $w(S)$ taken over the numerical semigroups $S \in NS(m,g)$, which corresponds to taking a weighted sum over integer points in $P_{m,g}$. Dividing by the number of integer points in $P_{m,g}$ gives the average weight of numerical semigroups containing $m$ with genus $g$. Lastly, we divide the average weight by $g^2$. The data suggests that these values seem to converge as $g$ increases.  We note that it is not difficult to compute the average weight in the case $m=3$ and show that the corresponding value in the figure should converge to $\frac{5}{18}$, which matches the data very well.

        \begin{figure}[htbp]
        \centering
            \includegraphics[width=0.618\textwidth]{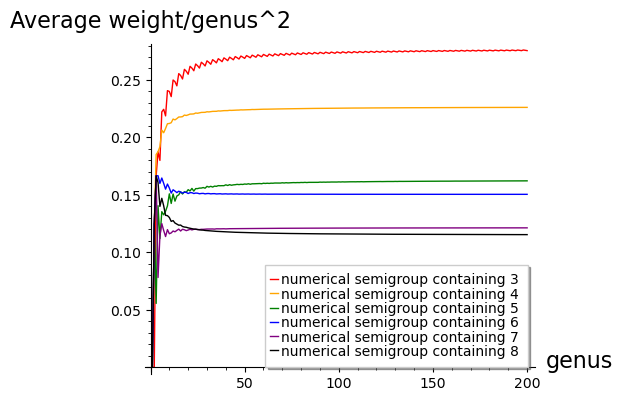}
        \caption{Curve plot of the quotient of average weight and genus square.}
        \label{fig:averageplot}
        \end{figure}

\section*{Acknowledgements}
We are truly grateful to Ardila who introduced us to his results with Brugall\'e that we have here extended to quasi-polynomial weights. Part of this work was done while most of the authors visited the American Institute of Mathematics in Palo Alto. The second author was supported by NSF Grant DMS 2142656.  The third author was supported by NSF Grant DMS 2154223.

\bibliographystyle{siam}
\bibliography{sample}

\end{document}